\documentclass{amsart}
\usepackage{amsmath}
\usepackage{amsxtra}
\usepackage{amscd}
\usepackage{amsthm}
\usepackage{amsfonts}
\usepackage{amssymb}
\usepackage{mathdots}
\usepackage{bbm}
\usepackage{graphicx}
\usepackage{longtable,booktabs}
\usepackage{cite}
\usepackage{tikz-cd}
\usepackage[bookmarksnumbered, colorlinks, plainpages]{hyperref}
\usepackage[utf8]{inputenc}
\usepackage{color}
\usepackage{enumerate}
\newtheorem{theorem}{Theorem}[section]

\newtheorem{lemma}[theorem]{Lemma}
\newcommand{\Sp}{\mathrm{Sp}}
\newcommand{\C}{\mathbb{C}}

\newcommand{\Id}{I}
\newcommand{\half}{\tfrac12}
\newcommand{\quarter}{\tfrac14}

\newcommand{\N}{\mathbb N}
\newcommand{\Ocal}{\mathcal O}

\newcommand{\diag}{\operatorname{diag}}

\newcommand{\hol}{\mathrm{hol}}

\newcommand{\holo}{\mathcal{O}}

\newtheorem{Thm}{Theorem}[section]
\newtheorem{Cor}[Thm]{Corollary}

\newtheorem{proposition}[theorem]{Proposition}
\theoremstyle{definition}
\newtheorem{definition}[theorem]{Definition}
\newtheorem{remark}[theorem]{Remark}

\usepackage[paperwidth=17.10cm, paperheight=24.60cm, left=1.5cm, right=1.5cm, top=2cm, bottom=2cm]{geometry}

\usepackage[]{amsrefs}
\usepackage{etoolbox}
\patchcmd{\section}{\scshape}{\large}{}{}
\patchcmd{\subsection}{\bfseries}{\normalfont}{}{}
\patchcmd{\subsubsection}{\itshape}{\normalfont}{}{}

\makeatletter
\def\@setauthors{%
	\begingroup
	\def\thanks{\protect\thanks@warning}%
	\trivlist
	\centering\footnotesize \@topsep30\p@\relax
	\advance\@topsep by -\baselineskip
	\item\relax\normalsize 
	\author@andify\authors
	\def\\{\protect\linebreak}%
	\authors%
	\ifx\@empty\contribs
	\else
	,\penalty-3 \space \@setcontribs
	\@closetoccontribs
	\fi
	\endtrivlist
	\endgroup
}
\def\@settitle{\begin{center}%
		\baselineskip14\p@\relax
		\bfseries\large
		\@title
	\end{center}%
}
\makeatother

\usepackage{titlesec}
\titleformat{\section}
{\normalfont\large\bfseries}
{\thesection}{1em}{}
\titleformat{\subsection}
{\normalfont\large\bfseries}
{\thesubsection}{1em}{}
\titleformat{\subsubsection}
{\normalfont\large\bfseries}
{\thesubsubsection}{1em}{}

\usepackage{tikz,xcolor,hyperref}

\definecolor{lime}{HTML}{A6CE39}
\DeclareRobustCommand{\orcidicon}{%
	\begin{tikzpicture}
		\draw[lime, fill=lime] (0,0) 
		circle [radius=0.16] 
		node[white] {{\fontfamily{qag}\selectfont \tiny ID}};
		\draw[white, fill=white] (-0.0625,0.095) 
		circle [radius=0.007];
	\end{tikzpicture}
	\hspace{-2mm}
}
\newcommand{\orcidFrank}{\href{https://orcid.org/0000-0003-2272-8687}{\orcidicon}}
\newcommand{\orcidSon}{\href{https://orcid.org/0000-0002-9560-6392}{\orcidicon}}
\newcommand{\orcidFeng}{\href{https://orcid.org/0009-0003-4512-2888}{\orcidicon}}
\newcommand{\orcidDung}{\href{https://orcid.org/0000-0001-6552-724X}{\orcidicon}}

\begin{document}
	\title[Parametric decompositions into products of involutions ]{Parametric Decompositions into Products of Involutions}
    \author[Gaofeng Huang]{Gaofeng Huang \textsuperscript{0,$\natural$}}
	\author[Frank Kutzschebauch]{Frank Kutzschebauch\textsuperscript{1,$\S$,*}}
	\author[Tran Nam Son]{Tran Nam Son\textsuperscript{2,\ddag}}
	\author[Truong Huu Dung]{Truong Huu Dung\textsuperscript{2,\dag}}
	\thanks{The first author was funded by Schweizerischer Nationalfonds Postdoc.Mobility P500-$2\_239113$. The second author was partially funded by Schweizerischer Nationalfonds grant 10005432. The third and fourth authors were supported by the Vietnam National Foundation for Science and Technology Development (NAFOSTED) under Grant No. 101.04-2025.41.}

	\maketitle

	\begin{center}
		{\small 
        *Corresponding author\\
            \textsuperscript{0}Department of Mathematics, \\ Harvard University,\\
            Cambridge, USA\\
			\textsuperscript{1}Mathematical Institute,\\ University of Bern,\\ Bern, Switzerland\\
			\textsuperscript{2}Department of Mathematics,\\ Dong Nai University,\\ 9 Le Quy Don Str., Tam Hiep Ward,\\ Dong Nai City, Vietnam\\
			\textsuperscript{$\natural$}ghuang@math.harvard.edu,\\ \orcidFeng \url{https://orcid.org/0009-0003-4512-2888}\\
            \textsuperscript{$\S$}frank.kutzschebauch@unibe.ch,\\ \orcidFrank \url{https://orcid.org/0000-0003-2272-8687}\\
			\textsuperscript{\ddag}trannamson1999@gmail.com; sontn@dnpu.edu.vn,\\ \orcidSon \url{https://orcid.org/0000-0002-9560-6392}\\
			\textsuperscript{\dag}dungth0406@gmail.com; thdung@dnpu.edu.vn,\\ \orcidDung\url{https://orcid.org/0000-0001-6552-724X}
		}
	\end{center}

    	\begin{abstract}
		It is a well-known result of Gustafson, Halmos and Radjavi, dating back to 1976, that any matrix $A$ in $SL_n(\mathbb{C})$ is a product of at most 4 involutions. We consider a natural continuous and holomorphic parameter dependence of this result in the spirit of Vaserstein and Gromov.
		Our main result shows that null-homotopy characterizes exactly those matrices in the special linear group over rings of continuous or holomorphic functions that are finite products of involutions. We give an upper bound $I(n, d)$, depending only on the size $n$ of the matrices and the dimension $d$ of the parameter space, for the number of involutions needed in the factorization. Our upper bound is not optimal, however we show that for fixed $n$ the optimal upper bound tends to infinity as $d\to \infty$. In some cases we are able to determine the optimal bound.
	\end{abstract}

	\section{Introduction}

	It is a well-known result of Gustafson, Halmos and Radjavi, dating back to 1976, that any matrix $A$ in $SL_n(\C)$ is a product of at most 4 involutions \cite{MR0399284}. Here an involution is a matrix $B$ with $B^2 = Id$. Clearly an involution has determinant $+1$ or $-1$. If the matrix $A$ depends continuously on a parameter $t$ from a topological space $T$, or holomorphically on a parameter $x$ from a Stein space $X$, the situation is much more complicated. We prove that the null-homotopy of the corresponding map $f: T \to SL_n (\C)$ is a necessary and sufficient condition for it to be a product of involutions.
    
    \begin{Thm}\label{main-continuous} Let $T$ be a normal finite-dimensional topological space and denote by $\mathcal{C}(T)$ the ring of continuous complex-valued functions on $T$. An element of $SL_n (\mathcal{C}(T))$ is a product of continuous  involutions if and only if it is null-homotopic, i.e.\ the corresponding continuous map $f: T \to SL_n (\C)$ is homotopic to a constant map.
    \end{Thm}
    
    The same holds for holomorphic maps from a Stein space to 
    $SL_n(\C)$. We remind the reader that, by Grauert's Oka principle, null-homotopy through continuous maps is equivalent to null-homotopy through holomorphic maps.

\begin{Thm}\label{main-holomorphic} Let $X$ be a reduced Stein  space and denote by $\holo (X)$ the ring of holomorphic  functions on $X$. An element of $SL_n (\holo (X) )$ is a product of holomorphic involutions  if and only if it is null-homotopic, i.e.\ the corresponding  map $f: X \to SL_n (\C)$ is homotopic to a constant map.
    \end{Thm}
    
    The necessity of both results follows from the fact that, writing $Inv_n=\{ B \in GL_n (\C) : B^2 = I \}$ for the set of involutions and $Inv_n^\pm$ for its subsets of determinant $\pm1$, the inclusion $Inv_n^+ \hookrightarrow SL_n(\C)$ is null-homotopic, and so is the inclusion $Inv_n^- \hookrightarrow N \cdot SL_n(\C)$, where $N$ is the diagonal matrix with diagonal entries $(-1,1,1,\ldots,1)$; see Proposition~\ref{necessary}.
    
    For the sufficiency we combine Vaserstein's factorization result and, respectively, the solution of the Gromov--Vaserstein problem with a general result on the factorization of unitriangular matrices into products of involutions. Here a unitriangular factor means either an upper or a lower unipotent triangular matrix. In fact, we prove that there is a uniform bound on the number of factors, depending only on the dimension of the parameter space and the size of the matrix. 
    Finding the exact values of such uniform bounds is an interesting problem that is still not well understood. For an account of similar problems in the parametric factorization of matrices, see the recent survey \cite{MR5024280}. We prove that as the dimension of the parameter space grows to infinity, the minimal number of involutions needed to write such a parametric matrix also grows to infinity; see the end of Subsections~\ref{sufficient-cont} and \ref{sufficient-holo}.
    In the last section we determine the exact numbers for
    $SL_2$ over topological and Stein parameter spaces of dimension one. Surprisingly, for smooth open Riemann surfaces the number is smaller than for general (non-smooth) open Riemann surfaces.

\section{Necessity}

Let $S^+ = SL_n(\C)$ and $S^- = N \cdot SL_n(\C) = \{ M \in GL_n(\C): \det M = -1 \}$, where $N = \mathrm{diag}(-1,1,1, \ldots, 1)$. 
An involutive matrix $A \in Inv_n$ is diagonalizable by conjugation and can have only $\pm 1$ as eigenvalues. 
Denote by $k$ the multiplicity of the eigenvalue $-1$ for $A$. Let $\Omega_k$ be the conjugacy class of $A$. This gives a partition of the involutions
\begin{align} \label{Inv-decompos}
   Inv_n = \dot{\bigcup}_{k=0}^n \Omega_k, 
\end{align}
where $\Omega_0 = \{ I \}$ and $\Omega_n = \{ -I \}$. 
\begin{proposition} \label{necessary}
    For each $k = 0, 1, \dots, n$, the inclusion $\Omega_k \hookrightarrow S^+$ is homotopic to the constant map $A \mapsto I$ for even $k$ and the inclusion $\Omega_k \hookrightarrow S^-$ is homotopic to $A \mapsto N$ for odd $k$, where $N = \mathrm{diag}(-1,1,1, \ldots, 1)$.
\end{proposition}

\begin{proof}
For each $k$, it suffices to find a continuous homotopy of continuous maps from $\Omega_k$ to $S^\pm$, which connects $A$ to an arbitrary constant matrix in $S^\pm$. Since $SL_n(\C)$ is path-connected, the inclusion map $\Omega_k \hookrightarrow S^\pm$ is then homotopic to the constant map $A \mapsto I$ for even $k$ and to $A \mapsto N$ for odd $k$. 

When $k$ is odd, $\det A=-1$. Consider the following homotopy 
\begin{align*}
    &H^-: \Omega_k \times [0,1] \to S^-, \\
    &  H^-(A,t) 
    =\frac{e^{it\pi \frac{n-2k}{2n}}}{2i} \Bigl[
    \bigl(e^{-\frac{i\pi}{2}(1-t)}+e^{\frac{i\pi}{2}(1-t)}\bigr) I
    -\bigl(e^{-\frac{i\pi}{2}(1-t)}-e^{\frac{i\pi}{2}(1-t)} \bigr) A
    \Bigr].
\end{align*}
The map is well-defined, since by conjugating $A$ to its diagonal form we get
\begin{align*}
    \det H^-(A, t) &= \left(\frac{e^{it\pi \frac{n-2k}{2n}}}{i} \right)^n \cdot 
    \left( e^{-\frac{i\pi}{2}(1-t)} \right)^k  \cdot
    \left( e^{\frac{i\pi}{2}(1-t)} \right)^{n-k} \\
    &= \frac{e^{i\pi \frac{n-2k}{2}}}{i^n} =  \frac{ \left( e^{ \frac{i\pi}{2}} \right)^{n-2k}}{i^n} = i^{-2k}=-1, \quad \forall t \in [0,1]. 
\end{align*}
Then, 
\begin{align*}
    H^-(A, 0) = A, \quad
    H^-(A, 1) = -i e^{i\pi \frac{n-2k}{2n}} I.
\end{align*}

When $k$ is even, $\det A = 1$. Consider the following homotopy 
\begin{align*}
    &H^+ : \Omega_k \times [0,1] \to S^+, \\
    &  H^+(A,t) = \frac{ e^{-it \pi \frac{2(n-k)-k}{n}}}{2} \Bigl[
    \bigl(e^{i\pi(1-t)}+e^{-2i\pi(1-t)}\bigr)I
    -\bigl(e^{i\pi(1-t)}-e^{-2i\pi(1-t)}\bigr)A
    \Bigr].
\end{align*}
Again by conjugating $A$ to its diagonal form we get 
\begin{align*}
    \det H^+(A, t) &= \left( e^{-it \pi \frac{2(n-k)-k}{n}} \right)^n \cdot 
    \left( e^{i\pi(1-t)} \right)^k  \cdot
    \left( e^{-2i\pi(1-t)} \right)^{n-k} \\
    &= e^{i \pi (2(n-k)-k)} =1, \quad \forall t \in [0,1], 
\end{align*}
where we use that $k$ is even in the last step. Moreover, 
\begin{align*}
    H^+(A, 0) = A, \quad H^+(A, 1) = e^{-i \pi \frac{2(n-k)-k}{n}} I. 
\end{align*}
This completes the proof. 
    \end{proof}

 \begin{Cor} Let a continuous map $f: T \to SL_n(\C)$ be a product of involutions, i.e.\ there are continuous maps $I_i: T \to Inv_n$, $i=1,2, \ldots , K$, with $$f (t) = I_1(t) \cdot   I_2 (t) \cdot \ldots \cdot  I_K(t).$$ Then $f$ is null-homotopic.
\end{Cor} 
\begin{proof}
    For $i= 1, \dots, K$, the continuous function $k_i : T \to \N, \, t \mapsto (n- \mathrm{tr}\, I_i(t))/2$ is locally constant. The fibers of $(k_1, \dots, k_K)$ form a finite partition of $T$. On each fiber $T'$, there exists for each factor $I_i$ a $k_i$ such that $I_i(T') \subset \Omega_{k_i}$. Since $\det f = 1$, the number of $i$'s with odd $k_i$ must be even. Applying Proposition \ref{necessary} to each $I_i$  we obtain the null-homotopy. 
\end{proof}

\section{Sufficiency} 
\subsection{The continuous decomposition} \label{sufficient-cont}
Our method is based on the factorization into unitriangular matrices. Although there is an abstract result of Laffey for unitriangular matrices (see Theorem \ref{Laffey}) which works over any commutative unital ring, we provide an improved version which holds for rings of complex-valued functions and has a better estimate for the number of involutive factors.

\begin{Thm}[Vaserstein \cite{MR0947649}*{Theorem 4}] \label{vaserstein}
    For any natural number $n$ and an integer $d \ge 0$ there exists a natural
    number $L = L(n, d)$ such that for any normal topological space $T$ of finite covering dimension $d$, every null-homotopic continuous map $f : T \to SL_n(\C)$ can be written as a product of no more than $L$ unitriangular matrices with entries in $\mathcal{C}(T)$. 
\end{Thm}

To our knowledge, very few bounds are known for these numbers $L = L(n, d)$; for instance, $L(n,1)=4$ follows from \cite{MR2822515}
and the fact that $\mathcal{C}(T)$ has Bass stable rank equal to $1$. Moreover, $L(n, d)\ge L(n+1, d)$ for all $n$ and $d$ \cite{MR0961333}. In a forthcoming paper we will provide bounds, linear in $d$, for all these numbers.

Let $R$ be a unital ring. Let $\mathcal{I}_n(R)$ denote the matrices in the elementary subgroup $E_n (R)$ that are finite products of involutions. Let $I(n,R)$ denote the minimal number such that every matrix in $\mathcal{I}_n(R)$ factorizes as a product of $I(n,R)$ involutions, and let $t(n,R)$ denote the minimal number such that every matrix in $E_n (R)$ is a product of $t(n,R)$ unitriangular factors.

Assume for a moment that every unitriangular matrix in $SL_n(R)$ is a product of $N$ involutions.
Since conjugates of involutions are involutions, a product of conjugates of $K$ unitriangular matrices will be a product of $NK$ involutions. This estimate can, however, be improved by the following reasoning.

Suppose we have a product of $3$ unitriangular matrices $U_1 U_2 U_3$ of alternating forms, e.g.\ $U_1, U_3$ are upper unitriangular and $U_2$ is lower unitriangular. Then 
\begin{align*}
    U_1 U_2 U_3 = (U_1 U_2 U_1^{-1}) (U_1 U_3) 
\end{align*}
is a product of $2$ conjugates of unitriangular matrices, i.e.\@ $2N$ involutions. This is because $U_2$ conjugated by $U_1$ is a product of $N$ involutions, while the product $U_1 U_3$ of upper unitriangular matrices is again upper unitriangular, hence again a product of $N$ involutions. Therefore, from $3$ unitriangular factors we get $2N$ involutions. An additional factor $U_4$ increases our estimate by one, but adding a further, fifth, factor does not increase the number of involutions any more:
\begin{align*}
    U_1 U_2 U_3 U_4 U_5 = (U_1 U_2 U_1^{-1}) (U_1 U_3 U_4 U_3^{-1} U_1^{-1}) ( U_1 U_3 U_5 ). 
\end{align*}
The last factor has the same form as $U_5$. This leads us to the following statement. 
\begin{proposition} \label{proposition:e-t-estimate} 
    If every unitriangular matrix is a product of $N$ involutions, then
    \begin{align*} 
    I(n, R)  \leq N \left(\left\lfloor \frac{1}{2} t(n,R) \right\rfloor + 1\right).
    \end{align*}
\end{proposition}

The following result of Laffey would settle our problem; we recall it for the reader who might be interested in more general rings.
\begin{Thm} [\cite{MR1635020}] \label{Laffey}
 Let $R$ be a unital ring. Then every unitriangular matrix in $SL_n (R)$ is a product of $10$ involutions.
\end{Thm}

However, we can improve the number of factors, starting with the following lemma.  The proof is similar to that of \cite[Lemma 2.1]{MR4815128}, where matrices over division rings are considered. For the convenience of the reader, we include brief proofs here.

\begin{lemma}\label{BiSo}
	Let $R$ be a unital ring containing an infinite field $F$ which is central in $R$, and let $n>1$ be an integer.
	Suppose that $A\in M_n(R)$ is an upper triangular matrix whose diagonal entries
	$a_1,\ldots,a_n$ lie in $F$ and are pairwise distinct.
	Then there exists a matrix $P\in GL_n(R)$ such that
	\[
	P^{-1}AP=\mathrm{diag}(a_1,\ldots,a_n).
	\]
\end{lemma}

\begin{proof}
	We argue by induction on $n$.
	
	First, consider the case $n=2$. Let
	\[
	A=\begin{pmatrix}
		a_1 & y\\
		0 & a_2
	\end{pmatrix},
	\]
	where $y\in R$.
	Since $a_1,a_2\in F$ and $a_1\neq a_2$, the element $a_1-a_2$ is invertible in $F$.
	Set $x=y(a_1-a_2)^{-1}$.
	A direct computation shows that
	\[
	\begin{pmatrix}
		1 & x\\
		0 & 1
	\end{pmatrix}
	A
	\begin{pmatrix}
		1 & x\\
		0 & 1
	\end{pmatrix}^{-1}
	=
	\begin{pmatrix}
		a_1 & 0\\
		0 & a_2
	\end{pmatrix}.
	\]
	
	Now assume that $n>2$ and that the statement holds for all matrices of size less than $n$.
	Write
	\[
	A=\begin{pmatrix}
		a_1 & \alpha\\
		0 & A'
	\end{pmatrix},
	\]
	where $A'\in M_{n-1}(R)$ is an upper triangular matrix whose diagonal entries
	are $a_2,\ldots,a_n\in F$, and $\alpha$ is a row vector of length $n-1$.
	
	Let $A''=\mathrm{diag}(a_2,\ldots,a_n)$.
	By the induction hypothesis, there exists $P\in GL_{n-1}(R)$ such that
	\[
	P^{-1}A'P=A''.
	\]
	Since the elements $a_1-a_i$ are nonzero and belong to $F$ for all $i\ge2$,
	the diagonal matrix $a_1 I_{n-1}-A''$ is invertible in $M_{n-1}(R)$.
	
	Set
	\[
	x'=-\alpha P (a_1 I_{n-1}-A'')^{-1}
	\]
	and define
	\[
	P'=
	\begin{pmatrix}
		1 & 0\\
		0 & P
	\end{pmatrix}
	\begin{pmatrix}
		1 & x'\\
		0 & I_{n-1}
	\end{pmatrix}.
	\]
	A straightforward computation shows that
	\[
	P'^{-1}AP'=
	\begin{pmatrix}
		a_1 & 0\\
		0 & A''
	\end{pmatrix}.
	\]
	This completes the proof.
\end{proof}

We are now in a position to establish the following result.

\begin{proposition}
	Let $R$ be a unital ring containing an infinite field $F$ which is central in $R$, and let $n>1$ be an integer.
	Then every unitriangular matrix in $M_n(R)$ can be expressed as a product of four involutions.
\end{proposition}

\begin{proof}
	Let $A$ be a unitriangular matrix in $M_n(R)$.
	Set
	\[
	t=\begin{cases}
		\dfrac{n}{2}, & \text{if $n$ is even},\\[0.3em]
		\dfrac{n-1}{2}, & \text{if $n$ is odd}.
	\end{cases}
	\]
	
	Since $F$ is infinite, we may choose elements
	$\alpha_1,\alpha_2,\ldots,\alpha_t \in F \setminus \{ 0, 1 \}$ such that 
    \[  \alpha_1,\alpha_1^{-1},\alpha_2,\alpha_2^{-1},\ldots,\alpha_t,\alpha_t^{-1} \text{ are pairwise distinct}.
    \]
	Write $$B=\begin{cases} \mathrm{diag}(\alpha_1,\alpha_1^{-1},\alpha_2,\alpha_2^{-1},\ldots,\alpha_t,\alpha_t^{-1}) & \text{ if $n$ is even};\\ \mathrm{diag}(\alpha_1,\alpha_1^{-1},\alpha_2,\alpha_2^{-1},\ldots,\alpha_t,\alpha_t^{-1},1) & \text{ if $n$ is odd}. \end{cases}$$
	
	A straightforward computation shows that, in both cases, the matrix $B$ can be written as the
	product of two involutions, each of which is block diagonal with $2\times2$ blocks of the form
	\[
	\begin{pmatrix}
		0 & \alpha_i\\
		\alpha_i^{-1} & 0
	\end{pmatrix}
	\quad \text{and} \quad
	\begin{pmatrix}
		0 & 1\\
		1 & 0
	\end{pmatrix},
	\]
	together with an additional $1\times1$ block equal to $1$ when $n$ is odd.
	Hence, $B$ itself is a product of two involutions. In particular, $B=CD$ where  $$C=\begin{cases} \mathrm{diag}\left(\begin{pmatrix}
			0 & \alpha_1\\
			\alpha_1^{-1} & 0
		\end{pmatrix},\begin{pmatrix}
		0 & \alpha_2\\
		\alpha_2^{-1} & 0
		\end{pmatrix},\ldots,\begin{pmatrix}
		0 & \alpha_t\\
		\alpha_t^{-1} & 0
		\end{pmatrix}\right) & \text{ if $n$ is even};\\ \mathrm{diag}\left(\begin{pmatrix}
			0 & \alpha_1\\
			\alpha_1^{-1} & 0
		\end{pmatrix},\begin{pmatrix}
			0 & \alpha_2\\
			\alpha_2^{-1} & 0
		\end{pmatrix},\ldots,\begin{pmatrix}
			0 & \alpha_t\\
			\alpha_t^{-1} & 0
		\end{pmatrix},1\right) & \text{ if $n$ is odd}, \end{cases}$$ and $$D=\begin{cases} \mathrm{diag}\left(\begin{pmatrix}
		0 & 1\\
		1 & 0
		\end{pmatrix},\begin{pmatrix}
		0 & 1\\
		1 & 0
		\end{pmatrix},\ldots,\begin{pmatrix}
		0 & 1\\
		1 & 0
		\end{pmatrix}\right) & \text{ if $n$ is even};\\ \mathrm{diag}\left(\begin{pmatrix}
		0 & 1\\
		1 & 0
		\end{pmatrix},\begin{pmatrix}
		0 & 1\\
		1 & 0
		\end{pmatrix},\ldots,\begin{pmatrix}
		0 & 1\\
		1 & 0
		\end{pmatrix},1\right) & \text{ if $n$ is odd}. \end{cases}$$
	
	Now observe that $A = B(B^{-1}A)$.
	By Lemma~\ref{BiSo} or its lower triangular version, the matrix $B^{-1}A$ is similar to $B^{-1}=DC$.
	Since similarity preserves the property of being a product of two involutions,
	it follows that $B^{-1}A$ is also a product of two involutions.
	Consequently, $A$ can be written as a product of four involutions, which completes the proof.
\end{proof}

Thus Proposition \ref{proposition:e-t-estimate} implies the following result.

\begin{Cor} \label{estimate}
 Let $L= L(n,d)$   be the number from  Theorem \ref{vaserstein}. Then 
 for any normal topological space $T$ of finite covering dimension $d$, every null-homotopic continuous map $f : T \to SL_n(\C)$ can be written as a product of no more than $4  (\left\lfloor L/2  \right\rfloor+ 1)$ 
 involutions  with entries in $\mathcal{C}(T)$. 
\end{Cor}

\begin{definition}
Let $I(n,d)$ denote the smallest number such that, for any normal topological space $T$ of finite covering dimension $d$, every null-homotopic matrix in $SL_n (\mathcal{C}(T))$ is a product of $I(n,d)$ continuous involutions.
\end{definition}

By a result of Dennis and Vaserstein, $\displaystyle\lim_{n \to \infty} L(n,d) \le 6$ \cite[Theorem 20]{MR0961333} (in fact this holds for any ring of finite Bass stable rank).  This gives 
$$\lim_{n \to \infty} I(n,d) \le 16.$$

On the other hand, there is a lower bound for $I(n,d)$ that grows linearly with $d$.

\begin{proposition}\label{prop:exp}
Let $R$ be $\mathcal{C}(T)$ (or $\mathcal O(X)$ in the holomorphic case) and $B\in GL_n(R)$ with $B^2=\Id$. Then $P:=\tfrac12(\Id-B)$ is an idempotent
in $M_n(R)$ and
\[
B=\exp(i\pi P).
\]
Consequently, if $f\in SL_n(R)$ is a product of $K$ involutions, then $f$ is a product of $K$ exponentials
$\exp(a_1)\cdots\exp(a_K)$ with $a_j\in M_n(R), j = 1, \dots,K$.
\end{proposition}

\begin{proof}
$P^2=\tfrac14(\Id-2B+B^2)=\tfrac12(\Id-B)=P$, and $\exp(i\pi P)=\Id+(e^{i\pi}-1)P=\Id-2P=B$.
\end{proof}

Combining this with the lower bound for factorizations into exponential factors from \cite{MR1305876}*{Theorem 1.10}, we conclude from Proposition \ref{prop:exp} the lower bound (the second term below) for the decomposition into involutions:
\begin{equation}\label{lower-bound}
    I(n,d)\ge \max \left\{ 4,   \ \left\lfloor \frac{d - 2}{2 (n^2 -1)} \right\rfloor + 1 \right\}.\end{equation}

In particular $$\lim_{d \to \infty} I(n,d) = \infty.$$

The first term in the bracket of the above inequality, namely 4, still requires a proof.  

When $n=2, d\ge 1$, the first part of the proof of Proposition~\ref{prop:continuous-rank-two}
gives the lower bound 4. For $n=2$ and $d=0$, consider  the union $\{0\}\cup\{1/m:m\geq1\}$ and the map 
\[
 A : \{0\}\cup\{1/m:m\geq1\} \to SL_2(\C), \quad t \mapsto \begin{cases}
     I_2 & t = 0 \\
     I_2+ t N_{1/t} &  t = 1/m
 \end{cases},
\]
where for $m \in \mathbb{N}$
\[
 N_{3m}=\begin{pmatrix}0&1\\0&0\end{pmatrix},\quad
 N_{3m+1}=\begin{pmatrix}0&0\\-1&0\end{pmatrix},\quad
 N_{3m+2}= \begin{pmatrix}-1&1\\-1&1\end{pmatrix}.
\]
Following the proof of Proposition~\ref{prop:continuous-rank-two}, we get the lower bound 4. 

For $n\geq3$, every matrix in $SL_n(\C)$ is a product of 4 involutions by Gustafson--Halmos--Radjavi \cite{MR0399284}, and for $n \neq 4$ this bound is sharp. Indeed, consider the example $\lambda I_n$ with $\lambda^n = 1$: if $\lambda I_n$ were a product of three involutions $J_1, J_2, J_3$, then $J_1 J_2 =  \lambda J_3$. A product of two involutions is similar to its inverse, thus $\lambda J_3$ is similar to $\lambda^{-1}J_3$. But then $\lambda^4 = 1$, which contradicts $\lambda^n=1$. 

By Kn{\" u}ppel and Nielsen \cite{MR1112667}, every matrix in $SL_4(\C)$ is a product of 3 involutions. For $d=0$, however, we can consider the following example, which is not a product of 3 involutions. Let $N_1, N_2, \dots, N_{15}$ denote the standard basis $E_{ij}$ of the Lie algebra $\mathfrak{sl}_4$ and set $N_{m+15}=N_m$ for $m \in \N$. Consider the null-homotopic continuous map 
\[
 A : \{0\}\cup\{1/m:m\geq1\} \to SL_4(\C), \quad t \mapsto \begin{cases}
     I_4 & t = 0 \\
     I_4 + t N_{1/t} &  t = 1/m
 \end{cases}.
\]
Assume that $A = J_1 J_2 J_3$ is a product of 3 continuous involutions. Then each $AJ_k$, for $k=1,2,3$, is a product of 2 involutions and thus similar to its inverse $J_k A^{-1}$. It follows that $\mathrm{tr} ((A-A^{-1})J_k)=0$ and pointwise $\mathrm{tr}(N_m J_k(1/m))=0$. Taking the limit along each subsequence $\{ m + 15 j\}_{j \in \N }$, we get $\mathrm{tr}(N_m J_k(0)) = 0$. Since $\{N_j\}$ form a basis of $\mathfrak{sl}_4$ whose orthogonal complement with respect to the trace pairing consists of multiples of the identity, $J_k(0) = \pm I_4$, since it is an involution. 

By continuity, $J_k(1/m)$ is close to $\pm I_4$ for large $m$, and thus $J_k(1/m) \mp I_4$ is invertible. Since $J_k$ is involutive, 
\[
    (J_k(1/m) + I_4 ) ( J_k(1/m) - I_4 ) = 0.
\]
Multiplying by the inverse then gives $J_k(1/m) = \pm I_4$ for large $m$, contradicting the definition of $A$.

\subsection{The holomorphic decomposition} \label{sufficient-holo}
In this section $X$ denotes a reduced Stein space and $f : X \to SL_n(\C)$ is a holomorphic map; equivalently, we consider $f$ as an element of $SL_n (\holo(X))$. 

Since holomorphic maps are continuous, the necessity of null-homotopy for writing $f$ as a product of holomorphic involutions follows from Proposition \ref{necessary}. For the sufficiency, we again use unitriangular factorization.

Manfred Klein and Karl Ramspott \cite[\S IV]{MR0966022} showed that every holomorphic map from a noncompact Riemann surface to $SL_n(\C)$ is a product of unitriangular factors. 
Gromov \cite{MR1001851} asked the following more general question: Does every holomorphic map $\C^m \to SL_n(\C)$ decompose into a finite product of holomorphic maps sending $\C^m$ into unipotent subgroups in $SL_n(\C)$?
He named it {\it the Vaserstein Problem}, after Vaserstein, who had established the analogous continuous factorization result. It was solved in greater generality by Ivarsson and Kutzschebauch. 

\begin{Thm}[Ivarsson--Kutzschebauch \cite{MR2874639}]\label{IK-SL}
    There exists a natural
    number $K = K(n, d)$ such that given any  reduced Stein space $X$ of finite complex dimension $d$, every null-homotopic holomorphic map $f : X \to SL_n(\C)$ can be written as a product of no more than $K$ unitriangular matrices with entries in $\holo (X)$. 
\end{Thm}
Regarding the number of factors, we first have the Bass-stable-rank-one case, i.e.\ that of noncompact (possibly singular) Riemann surfaces, where $K(n,1)=4$ for all $n \ge 1$. Again  from \cite{MR0961333} we have 
\[
    K(n,d) \le K(2,d) \text{ for all } n \ge 2, \, d \ge 0 
\]
and $\displaystyle\lim_{n \to \infty} K(n,d) \le 6$ from Dennis--Vaserstein \cite[Theorem 20]{MR0961333}. 
Also $K(n,2)=5$ was proven in \cite{MR2869067}.

As in the continuous case, we deduce the following.
\begin{Cor}
 Let $K= K(n,d)$   be the number from  Theorem~\ref{IK-SL}. Then, for
 any reduced Stein space $X$ of finite complex dimension $d$, every null-homotopic holomorphic map $f : X \to SL_n(\C)$ can be written as a product of no more than  $4 (\left\lfloor K/2  \right\rfloor+ 1)$ involutions  with entries in $\holo (X)$.
 \end{Cor}

Let $I(n,d)_{\hol}$ denote the minimal number such that, for any Stein space $X$ of dimension $d$, every null-homotopic matrix 
$A \in SL_n (\holo (X))$ factorizes as a product of $I(n,d)_{\hol}$ holomorphic involutions. Then we have just shown $$I(n,d)_{\hol} \le 4 (\left\lfloor K(n,d)/2  \right\rfloor+ 1).$$

Similarly to the continuous case, using \cite{BrudSasa} we can deduce:
$$\lim_{n \to \infty} I(n,d)_{\hol} \le 16  \  \text{ and } \lim_{d \to \infty} I(n,d)_{\hol} = \infty.$$ 

\subsection{The exact numbers in the case \texorpdfstring{$n=2$ and $d=1$}{n=2 and d=1}}

For $n=2$, we can determine the exact numbers of involutive factors; they read
\begin{equation}\label{eq:two-widths}
 I(2,1)=4,\qquad I(2,1)_{\hol}=4.
\end{equation}
By contrast, on a connected smooth open Riemann surface every matrix in
$SL_2(\Ocal(X))$ is a product of two holomorphic involutions; see Proposition \ref{prop:smooth-holomorphic} below.
\bigskip

\begin{remark} By the discussion after inequality~\eqref{lower-bound}, we 
see that $I(2,0)=4$. Since a zero-dimensional Stein space is a discrete set of points, $I(n,0)_{\hol}$ is the number over the field $\C$; it equals $2$ for $n=2$, $3$ for $n=4$, and $4$ for all other $n$.
\end{remark}
We first have an elementary observation.

\begin{lemma}\label{lem:odd-to-even}
Let $m\geq1$, let $S$ be a topological space, and let
$A: S\to SL_2(\C)$ be continuous.  If $A$ is a product of $2m+1$
continuous involutions, then it is a product of $2m$ continuous involutions.
The same holds for products of holomorphic involutions on a reduced complex space.
\end{lemma}

\begin{proof}
    By assumption $A=J_1\cdots J_{2m+1}$ is a product of an odd number of involutions. Since $\det A = 1$ and $\det J_i \in \{1, -1\}$, we have by continuity that $\det J_i = 1$ for some $i = 1, 2, \dots, m+1$. The only involutions in $SL_2(\C)$ are $I_2$ and $-I_2$. Hence, by continuity, $J_i = I_2$ or $J_i = -I_2$. 
\end{proof}

\begin{proposition}\label{prop:continuous-rank-two}
For $n=2$ and $d=1$ one has
$I(2,1)=4$.
\end{proposition}

\begin{proof}
For $\theta\in\mathbb R$ put
\[
 N(\theta)=
 \begin{pmatrix}
  -\sin\theta\cos\theta&\cos^2\theta\\
  -\sin^2\theta&\sin\theta\cos\theta
 \end{pmatrix}.
\]
Then $N(\theta)^2=0$.  We define a continuous null-homotopic map
by
\[
 A:[0,1]\to SL_2(\C), \quad t \mapsto \begin{cases}
     I_2 & t = 0 \\
     I_2+tN(1/t) & 0 < t \le 1
 \end{cases},
\]
which is homotopic to the constant identity map by $I_2+stN(1/t)$, $s\in [0,1]$. 

Suppose that $A=JK$ for continuous involutions $J,K : [0,1] \to SL_2(\C)$.  Then
$JAJ=A^{-1}$, and hence
\begin{equation}\label{eq:continuous-anticommute}
 J(t)N(1/t)+N(1/t)J(t)=0\quad \text{ for } t>0.
\end{equation}
Along sequences $\{ t_i \}$ tending to zero for which $1/t_i$ tends modulo $2\pi$ to
$0$, $\pi/2$, and $\pi/4$, respectively, continuity in
\eqref{eq:continuous-anticommute} shows that $J(0)$ anticommutes with
\[
 N(0)=\begin{pmatrix}0&1\\0&0\end{pmatrix},\qquad
 N(\pi/2)=\begin{pmatrix}0&0\\-1&0\end{pmatrix},\qquad
 2N(\pi/4)=\begin{pmatrix}-1&1\\-1&1\end{pmatrix}.
\]
Anticommutation with the
first two forces $J(0)=a\diag(1,-1)$, and anticommutation with the third then
yields $a=0$.  This contradicts $J(0)^2=I_2$.  Thus two factors do not suffice,
and Lemma~\ref{lem:odd-to-even} excludes three factors. Hence $I(2,1) \ge 4$.

For the reverse inequality, by Vaserstein \cite[Theorem 7]{Vaserstein1971} the Bass stable rank of the ring $\mathcal{C}(X,\C)$ of complex-valued continuous functions on a one-dimensional topological space $X$ is one. Thus, by Vavilov--Smolensky--Sury \cite[Lemma 1]{Vavilov}, every continuous map $A : X \to SL_2(\C)$ factorizes as 
\[
 A=L(g_1)U(g_2)L(g_3)U(g_4),
\]
where
\[
 L(g)=\begin{pmatrix}1&0\\g&1\end{pmatrix},\qquad
 U(g)=\begin{pmatrix}1&g\\0&1\end{pmatrix},
\]
then
\begin{equation}\label{eq:lulu-involutions-cont}
 A=(L(g_1)D)(DU(g_2))(L(g_3)D)(DU(g_4)) \quad \text{ with } D=\diag (1, -1).
\end{equation}
These factors are involutive, since $DL(g)D=L(g)^{-1}$ and $DU(g)D=U(g)^{-1}$.
\end{proof}

The same trick also gives an upper bound in the holomorphic case, which turns out to be optimal.

\begin{proposition}\label{prop:singular-holomorphic}
For $n=2$ and $d=1$ one has $I(2,1)_{\hol}=4$.
\end{proposition}

\begin{proof}
By Ivarsson and Kutzschebauch \cite[Theorem~5.1]{IKnumber}, every holomorphic
$A: X\to SL_2(\C)$ on a one-dimensional Stein space admits a factorization
\[
 A=L(g_1)U(g_2)L(g_3)U(g_4)
\]
with $g_j\in\Ocal(X)$. The same argument leading to Equation
\eqref{eq:lulu-involutions-cont} gives 4 holomorphic involutions.

To see that fewer factors do not work, consider the affine cusp
\[
 X_{\mathrm c}=\{(z,w)\in\C^2:w^2=z^3\}
\]
and the holomorphic matrices
\[
 B=\begin{pmatrix}-w&z\\-z^2&w\end{pmatrix},
 \qquad A=I_2+B.
\]
The equation of the cusp gives $B^2=0$, so $A\in SL_2(\Ocal(X_{\mathrm c}))$
and $A^{-1}=I_2-B$.  Moreover, $A$ is null-homotopic through the holomorphic
family $A_s=I_2+sB$, $0\leq s\leq1$.

At the cusp, the local ring is $\Ocal_{X_{\mathrm c},0}\cong\C\{t^2,t^3\}$ with $z=t^2, w=t^3$. 
If $A$ and $A^{-1}$ were locally holomorphically similar, there would exist an
invertible
\[
 H=\begin{pmatrix}p&q\\r&s\end{pmatrix}\in GL_2(\Ocal_{X_{\mathrm c},0}) \, \text{ such that } BH+HB=0.
\]
Direct computation gives $s=-p$ and $r=2tp+t^2q$.
Since $p,q,r\in\C\{t^2,t^3\}$, the second identity implies
$tp\in\C\{t^2,t^3\}$.  This ring contains only terms of order at least 2, so
$p(0)=0$. Consequently, $p,r\in (t^2,t^3)$.  It follows that $\det H=-p^2-qr\in (t^2,t^3)$, 
contrary to the invertibility of $H$.  

Hence, $A$ is not locally
holomorphically similar to $A^{-1}$ at the cusp;
in particular, it cannot be a
product of two involutions, since $A=J J'$ would imply $JAJ=A^{-1}$.
Lemma~\ref{lem:odd-to-even}, with $m=1$, excludes 3 factors, and the lower
bound 4 follows.
\end{proof}

In the absence of singularities, the number of involutive factors drops from 4 to 2.

\begin{proposition}\label{prop:smooth-holomorphic}
Let $X$ be a connected smooth open Riemann surface.  Every matrix  in  $SL_2(\Ocal(X))$ is a product of two holomorphic involutions in
$GL_2(\Ocal(X))$. 
\end{proposition}

\begin{proof}
Since a connected open Riemann surface $X$ has the homotopy type
of a one-dimensional CW-complex, whereas $SL_2(\C)\simeq S^3$ is simply
connected, every holomorphic map $A : X \to SL_2(\C)$ is null-homotopic as a continuous map and thus null-homotopic as a holomorphic map by Grauert's Oka principle.

Set
\[
 \tau=\frac{\operatorname{tr}A}{2},\qquad
 B=A-\tau I_2=\begin{pmatrix}e&b\\c&-e\end{pmatrix}.
\]
By the Cayley--Hamilton identity for $A$, we have
\begin{equation}\label{eq:CH-traceless}
 B^2=(\tau^2-1)I_2.
\end{equation}
If $B=0$, then $A=\pm I_2$, and $A=D(DA)$ is a product of two involutions. Here $D = \diag(1,-1)$.

Assume $B\neq0$.  By Wedderburn's lemma, $\mathcal{O}(X)$ is a B{\'e}zout
domain.  Hence $(e,b,c)=g \mathcal{O}(X)$ for some nonzero
$g\in \mathcal{O}(X)$ and we can write
\[
 B=gC,\qquad C=\begin{pmatrix}e_0&b_0\\c_0&-e_0\end{pmatrix},
 \qquad (e_0,b_0,c_0)=\mathcal{O}(X).
\]
Thus $C(x)\neq0$ for every $x\in X$.  Set
\[
 q=e_0^2+b_0c_0=-\det C,
 \qquad C_q=\begin{pmatrix}0&q\\1&0\end{pmatrix}.
\]
Then $C^2=qI_2$.  At every $x\in X$, the nonzero traceless matrix $C(x)$ is
non-scalar and therefore cyclic.  Choosing a local cyclic vector $v$, the
matrix $P_U=(v,Cv)$ is holomorphically invertible near $x$ and satisfies
\[
 P_U^{-1}CP_U=C_q.
\]
Thus $C$ and $C_q$ are locally holomorphically similar.  Guralnick's
local-to-global similarity theorem \cite{Guralnick} gives
$P\in GL_2(\Ocal(X))$ with $P^{-1}CP=C_q$.  From
\eqref{eq:CH-traceless} we obtain $\tau^2-1=g^2q$, and hence
\[
 A_0:= P^{-1}AP = \tau I_2 + g C_q
   = \begin{pmatrix}\tau&gq\\g&\tau\end{pmatrix}.
\]
Since $\det A_0=1$, we get $DA_0D=A_0^{-1}$. 
Therefore
\begin{align*}
 A=(PDP^{-1})\bigl(P(DA_0)P^{-1}\bigr),
\end{align*}
where both factors are holomorphic
involutions of determinant $-1$.
\end{proof}


\end{document}